\documentclass[11pt,a4paper]{article}

\usepackage{a4wide}
\usepackage{graphicx}
\usepackage{latexsym}
\usepackage{epsfig}
\usepackage{amssymb}
\usepackage{amstext}
\usepackage{amsgen}
\usepackage{amsxtra}
\usepackage{amsgen}
\usepackage{amsthm}

\newtheorem{thm}{Theorem}[section]
\newtheorem{prop}[thm]{Proposition}
\newtheorem{lemma}[thm]{Lemma}

\theoremstyle{definition}

\numberwithin{equation}{section}

\newcommand{\N}{\mathbb{N}}

\newcommand{\R}{\mathbb{R}}

\begin{document}
\title{Network structured kinetic models of social interactions }
\author{Martin Burger\thanks{Department Mathematik,  Friedrich-Alexander Universit\"at Erlangen-N\"urnberg, Cauerstr. 11, D 91058 Erlangen, Germany. e-mail: martin.burger@fau.de  } }
\maketitle
\begin{abstract}
The aim of this paper is to study the derivation of appropriate meso- and macroscopic models for interactions as appearing in social processes. There are two main characteristics the models take into account, namely a network structure of interactions, which we treat by an appropriate mesoscopic description, and a different role of interacting agents. The latter differs from interactions treated in classical statistical mechanics in the sense that the agents do not have symmetric roles, but there is rather an active and a passive agent. We will demonstrate how a certain form of kinetic equations can be obtained to describe 
such interactions at a mesoscopic level and moreover obtain macroscopic models from monokinetics solutions of those. 

The derivation naturally leads to systems of nonlocal reaction-diffusion equations (or in a suitable limit local versions thereof), which can explain spatial phase separation phenomena found to emerge from the microscopic interactions. We will highlight the approach in three examples, namely the evolution and coarsening of dialects in human language, the construction of social norms, and the spread of an epidemic. 

%
\end{abstract}

\section{Introduction}

The mathematical modelling of social interactions has been a topic of high recent interest. In early years this was naturally a field of qualitative models, which could at most be used to explain few macroscopic statistical data, and hence the interest in computing detailed distributions or spatial dependencies was limited. With the propagation of the internet and in particular the wide spreading of social networks, the field is changing significantly in recent years, since suddenly there is a huge amount of data to which models and predictions can be compared. This change is accompanied with increasing computational power, which allows for microscopic simulations, the corresponding field of agent-based models is of increasing importance within the social sciences and related fields like history or linguistics (cf. e.g. \cite{axelrod,heppenstall,klein}). 

From a mathematical point of view, it is natural to approach the transition from microscopic interaction to macroscopic models with the methods of statistical physics and kinetic theory, yielding (systems of) partial differential equations for distributions, with well-established asymptotic methods to further simplify or to analyze pattern formation (cf. \cite{bardos,bouchut,cercignani,cercignani2,degond,glassey,golse,levermore,pareschitoscani}). Such approaches have been used recently with success to explain macroscopic distributions in socio-economic interactions (cf. e.g. \cite{burgercaffarelli,fraiatosin,helbing,naldi,pareschitoscani,pareschitoscani2,toscanitosin,tosinzanella}) as well as several aspects of opinion formation and polarization (cf. e.g. \cite{boudinsalvarani,brugnatoscani,duering,toscani}). 

An issue that is naturally built into social processes is the network structure of interactions, which is commonly modelled via random networks. In meso- and macroscopic limits the network structure is usually lost, only few general characteristics of the network models feed into the remaining equations. From a rigorous point of view, the network limit poses particular challenges that are only partly resolved (cf. \cite{coppini,delattre}). In this paper we want to take another route towards incorporating a certain network structure into meso- and macroscopic models.  We avoid to describe th detailed network structure of the $N$-particle (resp. agent) system, but rather describe the agent by a structural variable $x$ (that can be considered as the spatial variable), which describes the position of the agent within a network. Then we can associate to each agent at $x$ and a second agent at $x'$ a rate of interaction and consider the agents overall as indistinguishable in the larger space of configurations consisting of their state and position, which reduces the mean-field limit to a standard setting. The limit then yields a kinetic equation with an additional structural variable. 

When the structural variable $x$ is viewed as a spatial variable, there is an interesting connection to the classical modelling of spatial systems with interactions such as reaction-diffusion equations (possibly with nonlocal diffusion). The key difference is that in our approach the particles do not change their position $x$ when interacting with others, while in classical kinetic theory interactions are local in space and only appear when the particles change position. We will comment on the differences between these models and implications to macroscopic equations. Let us mention that with respect to social behaviour the classical approach is very natural for face-to-face interactions, thus basically the only relevant one up to the last century. In modern digital communication, a fully nonlocal interaction without any change of position seems to be more relevant however. 

We will highlight our approach by case studies for three applications. The first concerns the propagation of dialects, which has been a topic a strong recent interest (cf. e.g. \cite{burridge1,burridge2,grieve,glaser,loreto,vazquez}). Using different arguments, a macroscopic PDE model has been obtained by Burridge \cite{burridge1,burridge2}, which is remarkably successfull in explaining the dialect maps in different countries as well as their coarsening. Naturally the structural variable is space and the rate of interactions between $x$ and $x'$ is a frequency to communication, still maximal at close distance. We will show that through the basic assumptions on the interactions made in \cite{burridge1}, we can derive a structured kinetic model. The original model by Burridge is recovered as a monokinetic solution of a Vlasov equation approximation of this model. Moreover, we show that solutions of the Vlasov model converge exponentially to monokinetic solutions in a Wasserstein-type metric. It is apparent that the original kinetic model, respectively a second-order Fokker-Planck type approximation, include further information about the stochastic nature of the process.  
The second case study concerns the emergence of social norms, which has mainly been studied by agent-based models and their direct simulation (cf. e.g. \cite{gavin,hawkins,pereda}). We use a recent agent-based model proposed by Shaw \cite{shaw} and perform analogous reasoning. The properties of the meso- and macroscopic models can easily be used to understand the long-time evolution and the emergence of segregation and coarsening effects as observed in simulations. The final one concerns the spread of an epidemic, actually a classical topic of mathematical modelling, for which we use a novel approach motivated by recent findings on pandemic spread. 

As another consequence of these exemplary cases we work out a frequent very asymmetric nature of interactions. There is an active and a passive agent in each interaction. The active agent chooses some action with some probability depending on the state, but does not change his state due to the interaction. Vice versa the passive agent only changes his state based on the action of the active one. Our approach can be used to derive  macroscopic models and study pattern formation or phase separation effects in such models. 

The overall organization of the paper is as follows: in Section 2 we introduce the basic modelling approach and derive mean-field models via a hierarchy of marginals. We also discuss further approximations for small change in the interactions. In Section 3 we investigate a model for the evolution of dialects (respectively phonetic variables) and show that our approach reproduces the previously proposed macroscopic model by Burridge \cite{burridge1} as a monokinetic solution in a natural Vlasov approximation. We also show that for this case there is a decay of variance in the Vlasov approximation, i.e. the monokinetic solutions approximate the overall dynamics well. In Section 4 we investigate a model for the construction of social norms, which has been studied by agent-based simulations previously. We demonstrate that the model fits into our framework in an analogous way and derive a macroscopic model of similar structure to the dialect model. In Section 5 we discuss a network structured model for the spread of a pandemic disease, which highlights the similarities and differences to conventional nonlocal reaction-diffusion models. We finally conclude and present several open questions in Section 6.

\section{Network-structured Kinetic Equations}

The setup in this paper is as follows: We start with a system of $N$ particles (synonymously called agents), each described by a structural variable $x_i \in \R^d$ and a state $v_i \in \R^s$ for $i=1,\ldots,N$. We will use the notation $z_i =(x_i,v_i) \in \R^{d+s}$ for the phase-space variable. Contrary to the classical transport model, we assume the position to be fixed, i.e., $\frac{dx_i}{dt}  = 0$, it is just used to encode a weighting of the particle interactions with a rate $\omega^N(x_i,x_j)$. Interactions between particles $i$ and $j$ are not assumed to be necessarily symmetric respectively conservative, but instead there is often an active and a passive role.  Let us start with a general two particle interaction between the $i$-the and the $j$-the particle
\begin{align}
v_i' &= v_i + a  ,\\
v_j' &= v_j + b  , 
\end{align} 
where $a$ and $b$ are random variables chosen from a joint probability distribution $\nu_{z_i,z_j}$ on $\R^s \times \R^s$. 

The particularly relevant case for social interactions, as we shall also see in case studies below, is when one of the particles assumes an active and the other a passive role. An example is an opinion expressed by the active particle leading to a change of opinion of the passive one. Often $\nu_{z_i,z_j}$  is concentrated at zero in $a$, refelecting that the active agent does not change state.
  %
Given some initial distribution of the particles, the can be described via the probability measure $\mu_N(\cdot;t)$ on $\R^{(d+s)N}$ for each $t \in \R_+$, whose evolution is governed by 
\begin{equation} \label{Nparticleweak}
\frac{d}{d t} \int \psi(Z) \mu_{N}(dZ_N;t) = \sum_{i \neq j} \int\int \omega^N\left(x_{i}, x_{j}\right) \left(\psi(Z_N^{a,b,i,j})-\psi(Z_N)\right)   ~\nu_{z_i,z_j}(da,db) ~\mu_{N}(dZ_N;t) 
\end{equation}
for each smooth test function $\psi$ on $\R^{(d+s)N}$, $Z_N = (z_1,\ldots,z_N)$, and $Z_N^{a,b,i,j}$ being a modified version of $Z_N$ with $v_i$ and $v_j$ changed to $v_i+a$, $v_j+b$. 
The right-hand side of \eqref{Nparticleweak} can be simplified using 
\begin{align*}
&\int\int \omega^N\left(x_{i}, x_{j}\right) \left(\psi(Z_N^{a,b,i,j})-\psi(Z_N)\right)   ~\nu_{z_i,z_j}(da,db) ~\mu_{N}(dZ_N;t) \\
& \qquad = \int \omega^N\left(x_{i}, x_{j}\right) \left(\int \psi(Z_N^{a,b,i,j})   ~\nu_{z_i,z_j}(da,db) - \psi(Z_N) \right) \mu_{N}(dZ_N;t)
\end{align*}

We can verify the well-posedness of this evolution equation as in the Picard-Lindel\"of Theorem on the space of Radon measures together with conservation of mass and nonnegativity) if $\omega^N$ and $\nu_{z_i,z_j}$ depend continuously on the states:
\begin{thm}
Let $\mu_N^0 \in {\cal P}(\R^{(d+s)N})$, $\omega^N \in C_b( \R^{2d})$, and $\nu \in C_b(\R^{d+s} \times \R^{d+s};{\cal P}(\R^{2s}))$. 
Then there exists a unique solution 
$$ \mu_N \in C^1(\R_+; {\cal P}(\R^{(d+s)N}))$$ 
of \eqref{Nparticleweak} with initial value $\mu_N^0$.
\end{thm} 
\begin{proof}
The evolution is of the form
$$ \partial_t \mu_N = {\cal L}^* \mu_N ,$$
with ${\cal L^*}: {\cal P}(\R^{(d+s)N}) \rightarrow {\cal P}(\R^{(d+s)N})$ defined as the adjoint of
$${\cal L}: C_b(\R^{(d+s)N}) \rightarrow C_b(\R^{(d+s)N}), \quad
\psi \mapsto \sum_{i \neq j} \int \omega^N\left(x_{i}, x_{j}\right) \left(\psi(Z_N^{a,b,i,j})-\psi(Z_N)\right)   ~\nu_{z_i,z_j}(da,db).$$
With our assumptions on $\nu$ and $\omega^N$ the linear operator ${\cal L}$ is well-defined and bounded and so is ${\cal L}^*$. Thus, the Picard-Lindel\"of theorem immediately yields existence and uniqueness of a solution in the Banach space of Radon measures. The fact that $\mu^N$ preserves mass one in time follows immediately with $\psi \equiv 1$ and the preservation of nonnegativity follows with the form
\begin{align*} & \frac{d}{d t} \left( e^{\lambda(N^2-1)t}\int \psi(Z) \mu_{N}(dZ_N;t) \right)= \\ & \qquad e^{\lambda(N^2-1)t} \sum_{i \neq j}   
\int \omega^N\left(x_{i}, x_{j}\right) \int \psi(Z_N^{a,b,i,j})   ~\nu_{z_i,z_j}(da,db) \mu_{N}(dZ_N;t) + \\
& \qquad e^{\lambda(N^2-1)t} \sum_{i \neq j}  \int  (\lambda- \omega^N\left(x_{i}, x_{j}\right) ) \psi(Z_N) \mu_{N}(dZ_N;t)    \end{align*}
and $\lambda \geq \Vert \omega^N \Vert_\infty$.
\end{proof}

Let us mention that alternatively we can derive existence and uniqueness in the $W_1$-Wasserstein metric (equivalent to the bounded Lipschitz metric) 
$$ d_{W_1}(\mu_N;\tilde \mu_N) = \sup_{\psi \in C^{0,1}, \Vert \psi \Vert \leq 1} \int \psi(z) \mu_N(dz) - \int \psi(z) \tilde \mu_N(dz) , $$
if $\omega^N$ and the maps 
$$ (z,\tilde z) \mapsto \int \vert a \vert \nu_{z,\tilde z}(da,db), \quad (z,\tilde z) \mapsto \int \vert b \vert \nu_{z,\tilde z}(da,db) $$
are Lipschitz-continuous.

%

\subsection{Mean-Field Limit}

In the following we shall derive a mean-field limit of the evolution equation \eqref{Nparticleweak} using a key assumption on the weights, namely
\begin{equation}
\omega^N(x,y) = \frac{1}N w(x,y) \qquad \forall (x,y) \in \R^d \times \R^d,
\end{equation}
which is a natural scaling. Then \eqref{Nparticleweak} reads
\begin{equation}  
\frac{d}{d t} \int \psi(Z) \mu_{N}(dZ_N;t) = \frac{1}N \sum_{i \neq j} \int\int w\left(x_{i}, x_{j}\right) \left(\psi(Z_N^{a,b,i,j})-\psi(Z_N)\right)   ~\nu_{z_i,z_j}(da,db) ~\mu_{N}(dZ_N;t) 
\end{equation}
and we can derive in a standard way a BBGKY-type hierarchy for the marginals
\begin{equation}
\mu_{N:k} = \int \ldots \int \mu_N(dz_{k+1} \ldots dz_N). 
\end{equation} 
Equations for the marginals can be derived easily when using a test function $\psi_k$ that only depends on 
$Z_k =(z_1,\ldots,z_k)$, which yields for $k=1,\ldots,N$
\begin{align*}
& \frac{d}{d t} \int \psi_k(Z_k) \mu_{N:k}(dZ_k;t) = \\ & \qquad \frac{1}N \sum_{1 \leq i \neq j \leq k} \int\int w\left(x_{i}, x_{j}\right) \left(\psi_k(Z_k^{a,b,i,j})-\psi_k(Z_k)\right)   ~\nu_{z_i,z_j}(da,db) ~\mu_{N:k}(dZ_k;t) +  \\
& \qquad \frac{N-k}N \sum_{i=1}^k  \int\int w\left(x_{i}, x_{k+1}\right) \left(\psi_k(Z_{k}^{a,i})-\psi_k(Z_{k})\right)   ~\nu_{z_i,z_{k+1}}(da,db) ~\mu_{N:k+1}(dZ_{k+1};t) + \\
& \qquad \frac{N-k}N \sum_{i=1}^k  \int\int w\left(x_{i}, x_{k+1}\right) \left(\psi_k(Z_{k}^{b,i})-\psi_k(Z_{k})\right)   ~\nu_{z_{k+1},z_i}(da,db) ~\mu_{N:k+1}(dZ_{k+1};t).
\end{align*}
Here we use the notation $Z_{k}^{a/b,i}$ for a version of $Z_k$ with $z_i$ changed to $z_i + a/b$. In the infinite limit $N\rightarrow \infty$ we formally arrive at the infinite hierarchy
\begin{align*}
& \frac{d}{d t} \int \psi_k(Z_k) \mu_{\infty:k}(dZ_k;t) = \\  
& \qquad  \sum_{i=1}^k  \int\int w\left(x_{i}, x_{k+1}\right) \left(\psi_k(Z_{k}^{a,i})-\psi_k(Z_{k})\right)   ~\nu_{z_i,z_{k+1}}(da,db) ~\mu_{\infty:k+1}(dZ_{k+1};t) + \\
& \qquad  \sum_{i=1}^k  \int\int w\left(x_{i}, x_{k+1}\right) \left(\psi_k(Z_{k}^{b,i})-\psi_k(Z_{k})\right)   ~\nu_{z_{k+1},z_i}(da,db) ~\mu_{\infty:k+1}(dZ_{k+1};t)
\end{align*}
for $k \in \N$.

\subsection{Kinetic Equation}

The infinite hierarchy of marginals allows for a solution in terms of product measures $\mu_{\infty:k} = \mu^{\otimes k}$, which characterizes the mean-field limit. The single-particle measure $\mu=\mu_{\infty:1}$ solves the kinetic equation
\begin{align*}
\frac{d}{d t} \int \varphi(z) \mu (dz;t) =& \int\int\int w\left(x,\tilde x\right) 
\left(\varphi(z^a)-\varphi(z)\right)   ~\nu_{z,\tilde z}(da,db)~\mu(d\tilde z;t)  ~\mu(dz;t) 
+ \\ & \int\int\int w\left(x,\tilde x\right) 
\left(\varphi(z^b)-\varphi(z)\right)   ~\nu_{\tilde z,z}(da,db)~\mu(d\tilde z;t)  ~\mu(dz;t) 
\end{align*}

We define $\eta$ as the projection of the measure $\mu$ to the structural variables, i.e.,
\begin{equation}
\eta(\cdot;t) = \int \mu(\cdot,dv;t) .  \label{eq:nudefinition}
\end{equation}
By using test functions of the form $\varphi(z) = \psi(x)$ we immediately see 
$$ \frac{d}{d t} \int \psi(x) \eta(dx;t) = \frac{d}{d t} \int \varphi(z) \mu (dz;t) = 0, $$
thus it is straight-forward to show the following result:
\begin{lemma}
Assume $\eta \in C(0,T;{\cal P}(\R^d))$ is well-defined by \eqref{eq:nudefinition}. Then $\eta$ is stationary, i.e. 
$\eta(\cdot;t) = \eta(\cdot;0)$ for all $t \in [0,T]$.
\end{lemma}

The stationarity of $\eta$ is a natural consequence of our modelling assumption that the network does not change. 
For higher moments we do not get equally simple results, e.g. the evolution of the first moment in $v$ is determined by 
\begin{align*}
\frac{d}{d t} \int v \mu (dz;t) =& \int\int W(z,\tilde z) 
 ~\mu(d\tilde z;t)  ~\mu(dz;t) 
 \end{align*}
with
$$ W(z,\tilde z)  = w\left(x,\tilde x\right)  \int a ~\nu_{z,\tilde z}(da,db) + \int b ~\nu_{\tilde z,z}(da,db). $$

\subsection{Vlasov Approximation}
 
As usual for kinetic equations we can proceed to local approximations if the changes $a$ and $b$ are small. If their higher order moments are negligible compared to the expectations, we can proceed in a straight-forward way to a Vlasov approximation, which is given in weak formulation
\begin{align}
\frac{d}{d t} \int \varphi(z) \mu (dz;t) =& \int\int  
\nabla_v \varphi(z) K(z,\tilde z)~\mu(d\tilde z;t)  ~\mu(dz;t) \label{eq:transportweak}
\end{align}
with 
\begin{equation}
K(z,\tilde z) = w(x,\tilde x) \left(\int a   ~\nu_{z,\tilde z}(da,db) + \int b ~\nu_{\tilde z,z}(da,db) \right). \label{eq:Kdefinition}
\end{equation} 
In order to perform suitable asymptotic analysis it is more convenient to assume that there is a small parameter $\epsilon$ scaling the interactions and
$$ \int (v^a - v)   ~\nu_{z,\tilde z}(da,db) + \int (v^b - v)  ~\nu_{\tilde z,z}(da,db)  = {\cal O}(\epsilon^\alpha)$$
for some $\alpha > 0$ and rescale time by $\epsilon^\alpha$ as well. Then we obtain instead \eqref{eq:transportweak}
with
\begin{equation}
K(z,\tilde z) = \lim_{\epsilon \downarrow 0} ~ \epsilon^{-\alpha} ~w(x,\tilde x) \left(\int a   ~\nu_{z,\tilde z}(da,db) + \int b  ~\nu_{\tilde z,z}(da,db) \right). \label{eq:Kdefinitioneps}
\end{equation} 

The corresponding strong form is 
\begin{equation} 
\partial_t \mu + \nabla_v \cdot (  \mu \int K(z,\tilde z)~\mu(d\tilde z;t) ) =0, \label{eq:transportstrong}
\end{equation}  
which is reminiscent of the classical Vlasov equation, however without transport term in $x$ and possibly a strong interaction in the $v$-space. it is well-known that the stability of this type of equations for sufficiently smooth $W$ can be derived in  Wasserstein metrics (cf. \cite{dobrushin,golse}) or equivalently via the method of characteristics (cf. \cite{braunhepp,golse,neunzert}).
In the case of \eqref{eq:transportstrong}, \eqref{eq:Kdefinition} the characteristic curves are given by the solutions of
\begin{equation}
\frac{dX}{dt}(z,t) = 0, \qquad \frac{dV}{dt}(z,t) = \int K(X(z,t),V(z,t),X(\tilde z,t),V(\tilde z,t)) \mu_0( d\tilde z)
\end{equation}
with initial values $X(z,0) = x$, $V(z,0)=v$. Due to the stationarity of $X$ we can compute $X(x,v,t)$ and formulate the characteristics solely in $V$ as 
\begin{equation}
\frac{dV}{dt}(z,t) = \int K(x,V(z,t),\tilde z,V(\tilde z,t)) \mu_0( d\tilde z).
\end{equation}
If $K$ is sufficiently regular, in particular Lipschitz with respect to $V$, the existence and uniqueness as well as stability estimated can again be obtained by ODE techniques.

Similar to the analysis along characteristics, we can also find a particular class of solutions correspoinding to monokinetic solutions in classical kinetic theory. 
Monokinetic solutions are of the form $\mu(dz;t) = \eta(dx) \otimes \delta_{V(x,t)}(dv), $
with a nonnegative Radon measure $\eta$ on $\R^d$ and $V$ being a solution of 
\begin{equation} 
\partial_t V(x,t) = \int K(x,V(x,t),\tilde x,V(\tilde x,t))~\eta(d\tilde x)   . 
\label{eq:monokinetic}
\end{equation}
Note that as before $\eta$ is a stationary measure, which is due to the stationarity of characteristics in $x$-space. Under suitable properties of the kernel $K$ (respectively the measure $\nu$) we may find exponentially fast convergence of solutions of the Vlasov equation to monokinetic ones with initial values 
$  V(x,0) = \int v \mu_0(dz)$, as we shall see in case studies below. Then the equation \eqref{eq:monokinetic} for $V$ is the relevant one to understand the dynamics and possible pattern formation. Indeed it is an interesting nonlinear and nonlocal equation, which can yield rich dynamics such as phase separation and coarsening, again illustrated below in examples. In other cases it can be relevant to study the full network-structured kinetic equation respectively its Vlasov approximation. 
 
\subsection{Fokker-Planck Approximation}

For a better approximation of the variance the second moment can be included in order to obtain a nonlinear and nonlocal Fokker-Planck equation, in weak form
\begin{align}
\frac{d}{d t} \int \varphi(z) \mu (dz;t) =& \int\int  
\nabla_v \varphi(z) K(z,\tilde z)~\mu(d\tilde z;t)  ~\mu(dz;t)  \nonumber
\\& \int\int  
\nabla_v^2 \varphi(z) : A(z,\tilde z)~\mu(d\tilde z;t)  ~\mu(dz;t)  \label{eq:fokkerplanckweak}
\end{align}
with $K$ defined by \eqref{eq:Kdefinition}, $:$ denoting the Frobenius scalar product
\begin{equation}
A(z,\tilde z) = \frac{1}2 w(x,\tilde x) \left(\int a \otimes a   ~\nu_{z,\tilde z}(da,db) + \int b \otimes b ~\nu_{\tilde z,z}(da,db) \right). \label{eq:Adefinition}
\end{equation} 
In strong form the Fokker-Planck equation becomes
$$ \partial_t \mu + \nabla_v \cdot (  \mu \int K(z,\tilde z)~\mu(d\tilde z;t) ) = \nabla \cdot (   \nabla \cdot (   \mu \int A(z,\tilde z)~\mu(d\tilde z;t)).$$
We leave a detailed discussion of the analysis of the second order equation for typical interactions as considered in the examples later to future research.

\subsection{Variants: Discrete Structures or States} 

There are several variants of the model we have formulated above in a purely continuum setting. However, there are some variants of the model in discrete or semidiscrete settings. An obvious case is related to a finite set of structural variables $x \in \{x_1,\ldots,x_M\}$. This can be set up in an analogous way as the model above, choosing a measure of the form
$$ \mu(dx,dv;t) = \frac{1}M \sum_{i=1}^M \delta_{x_i}(dx) \lambda_i(dv;t). $$
The weights $w$ need to be specified only for the discrete values $x_i$. 

Another semidiscrete case concerns the state variables $v$, between which the agents oscillates. Such a model arises as a special case of our approach if $\nu_{z,\tilde z}$ and $\mu(\cdot,dv;t)$ are concentrated at a finite number of possible states, with transitions $a$ and $b$ such that this finite state space remains invariant. 

\section{Case Study: Evolution of Dialects}

The first model we study in our framework is related to the evolution of dialects, as discussed by Burridge \cite{burridge1}. We will rederive this model as a monokinetic solution to a network-structured equation, respectively a spatially local approximation.

The model by Burridge \cite{burridge1} is based on the memory of the way certain vowels are used in words, which change during interactions with others (i.e. hearing them speak). Under the assumption that there are $M$ ways to use that vowel, the memory of each agent is of the form ${\bf v} = (v_1,\ldots,v_M)$ being an element of the convex set
$$ K=\{v \in \mathbb{R}^M~|~v_i \geq 0, \sum v_j = 1 \}, $$
i.e. $v_i$ is perceived as a relative frequency of the appearance of certain words. When using the vowel in conversations, an agent with memory ${\bf v}$ will choose variant $i$ with probability $p_i({\bf v})$, the expression proposed in \cite{burridge1} is
\begin{equation} \label{eq:piburridge}
p_i({\bf v}) = \frac{v_i^\alpha}{\sum_j v_j^\alpha} 
\end{equation}
with $\alpha > 1$ in order to give stronger weight to those with highest memory. We will also use the notation 
\begin{equation} {\bf p}: K \rightarrow K, {\bf v} \mapsto (p_i({\bf v}))_i .\end{equation}
It can be shown that with the above choice of ${\bf p}$ with $\alpha \geq 1$ we obtain a monotone invertible map ${\bf p}: K \rightarrow K$ (cf. \cite{doss}).

The collisions are due to hearing a certain variant, with post-collisional memory
$$  {\bf v}_i'= \frac{1}{1+\gamma} {\bf v} + \frac{\gamma}{1+\gamma} {\bf e}_i$$
in case the speaker (active agent) has chosen variant $i$, with ${\bf e}_i$ being the $i$-th unit vector.
This leads to the following formula for the pre-collisional memory
$$ {\bf v}_i^*= (1+\gamma) {\bf v} - \gamma {\bf e}_i.$$
Here $\gamma > 0$ is a parameter related to the weight given to the last appearance compared to the long-term memory. It is natural to think of $\gamma$ as a small parameter, since a single appearance of a variant will have low impact. 

\subsection{Boltzmann Equation}

The dialect model can be put in a semidiscrete state setting corresponding to our general framework above. There are only $M$ different transitions possible, the total state space is however continuous due to the dependence of the transitions on ${\bf v}$. 
$$ \nu_{z,\tilde z}(da,db) = \delta_{0}(da) \otimes \sum_{i=1}^M p_i({\bf  v}) \delta_{\frac{\gamma}{1+\gamma} ({\bf e}_i-{\bf \tilde v})}(db)$$
%
Thus, the corresponding Boltzmann-type equation for the evolution of the measure $\mu$ is given by
\begin{equation}
\frac{d}{dt} \int \varphi(z) \mu(dz;t) = \sum_{i=1}^M 
\int   \int w(x,\tilde x) ( \varphi( z_i') - \varphi( z))  p_i({\bf \tilde v}) \mu(d\tilde z;t) \mu (dz;t) 
\end{equation}
with $z_i'=(x,{\bf v}_i')$. Note that due to $ \sum_{i=1}^M  p_i({\bf v}) = 1$ for all ${\bf v}$ we can simplify the loss term to
$$  \sum_{i=1}^M  \int   \int   w(x,\tilde x) \varphi( z)   p_i({\bf \tilde v}) \mu(d\tilde z;t) \mu (dz;t)  = 
\int   \kappa(  x) \varphi( z) \mu(d  z;t)$$
with
$$ \kappa(  x) = \int w(x,\tilde x)  \mu(d\tilde z;t) =  \int w(x,\tilde x) \rho(\tilde x)~d\tilde x . $$
Note that we used the stationarity of $\int \mu(\cdot,dv;t)$ to derive the stationary coefficient $\kappa$.

In the remainder of this chapter we shall assume that $\mu$ is absolutely continuous with respect to the Lebesgue measure on $\R^d \times K$ and write $\mu = f ~dz$ with a probability measure $f$. Moreover, we define the spatial density 
\begin{equation}
\rho(x,t)  = \int f(x,v,t)~dv  .  \label{eq:rhodefinition}
\end{equation}

For this special interaction we can also get an equation for the mean value
$$ V(x;t)   = \int v f(x,{\bf v};t)~d{\bf v} $$
by choosing $\varphi(z) = \psi(x) \otimes v$, namely
\begin{align*}
\frac{d}{dt} \int \psi(x) V(x,t)~dx =& \sum_{i=1}^M 
\int     w(x,\tilde x) \psi(x) (\frac{\gamma}{1+\gamma} ({\bf e}_i-{\bf \tilde v}) )  p_i({\bf \tilde v})f(\tilde z;t)d\tilde z f(z;t)~dz \\
=& \frac{\gamma}{1+\gamma} \sum_{i=1}^M 
\int  \int  \psi(x) w(x,\tilde x)  \rho(x) {\bf p}({\bf \tilde v}) f(\tilde z;t)d\tilde z  dx - \\ &
\frac{\gamma}{1+\gamma} \int \psi(x) \kappa(x) V(x,t) ~dx.
 \end{align*}
This a non-closed equation if ${\bf p}$ is nonlinear, in particular we see that ${\bf v}$ is not a collision invariant, which is due to the asymmetric structure of the interactions.

\subsection{Vlasov Approximation} 

Noticing that naturally  $\gamma$ is a small parameter we may perform a (formal) asymptotic as $\gamma \rightarrow 0$ in order to derive a Vlasov equation. As above, with an additional rescaling of time by $\frac{\gamma}{1+\gamma}$ we arrive at the equation in weak form
\begin{align*} 
&\frac{d}{dt} \int  \varphi(x,{\bf v}) f(x,{\bf v},t)~dz =   -
\int w(x,\tilde x)   \nabla_{\bf v } \varphi(x,{\bf v}) f(x,{\bf v},t)\int    f(\tilde x,{\bf \tilde v},t) ( {\bf v} - {\bf p}({\bf \tilde v}))  ~d\tilde z~dz
\end{align*}
respectively in strong form
\begin{equation}  \label{eq:dialecttransport}
 \partial_t f(x,{\bf v},t) = \nabla_{\bf v} \cdot \left( f(x,{\bf v},t)  \int w(x,\tilde x)  f(\tilde x,{\bf \tilde v},t) ( {\bf v} - {\bf p}({\bf \tilde v}))  ~d\tilde z  \right) ,
\end{equation}


\subsubsection*{Monokinetic Solutions}

The monokinetic solutions in the case of the dialect model are of the form $ \rho(x) \delta_{v - V(x,t)}, $
where $V$ solves the nonlocal equation
$$ \partial_t V(x,t) =   \int w(x,\tilde x)   ({\bf p}(V(\tilde x,t)) -   V(x,t) ) \rho(\tilde x) d\tilde x$$
It is instructive to rewrite the model as 
\begin{equation} \rho(x) \partial_t V(x,t) = - \rho(x) \kappa(x)  (V - {\bf p}(V)) + \int w(x,\tilde x)   \rho(x)\rho(\tilde x) ({\bf p}(V(\tilde x,t)) -  {\bf p}(V(x,t))dx, \label{eq:dialectmonokinetic}
\end{equation}
which highlights its structure as a nonlocal reaction-diffusion equation. The first term is a multistable reaction, and it is easy to figure out that its stable steady states are in the corners of $K$, which corresponds to phase separation. The second term is a nonlocal diffusion operator acting on ${\bf p}(V)$, which is known to promote coarsening behaviour as in the celebrated Allen-Cahn equation (cf. \cite{allencahn,duncan,kohnotto}). 

Let us mention an alternative modelling approach, which is more convenient in literature: assuming that agents interact only locally and move independently (with the same kind of kernel), we would obtain the more standard nonlocal reaction-diffusion model
\begin{equation} \rho(x) \partial_t V(x,t) = - \rho(x) \kappa(x)  (V - {\bf p}(V)) + \int w(x,\tilde x)   \rho(x)\rho(\tilde x)  V(\tilde x,t)-  V(x,t)dx.\label{eq:dialectreactiondiffusion}
\end{equation}
The key difference to \eqref{eq:dialectmonokinetic} is the linearity  of the nonlocal diffusion term, it remains to understand the implications. 

\subsubsection*{Concentration to Monokinetic Solutions}

In the following we investigate the concentration behaviour of solutions to \eqref{eq:dialecttransport} in the $v$-space. For this sake we compute
We denote the support of $\rho$ by $\Omega \subset \R^d$, assuming $\Omega$ is a regular domain, and perform all integrations with respect to $x$ on $\Omega$.
\begin{prop}
Let $f$ be a sufficiently regular weak solution of \eqref{eq:dialecttransport} and let 
$${\bf \overline{V}}(x,t)  = \frac{1}{\rho(x)} \int {\bf v} f(x,{\bf v};t) d{\bf v}  $$
denote the expectation of  ${\bf v}$. Then the quadratic variation in $v$, given by
$$ {\cal V}(t) = \int \int \vert{\bf v} - {\bf \overline{V}}(x,t) \vert^2 f(x,{\bf v},t)~dz,$$
is nonincreasing in time, in particular it is increasing as long $f$ is not concentrated at ${\bf \overline{V}}(x,t)$ on the support of $\kappa$. If there exists a positive constant $\kappa_0$ such that $\kappa_0 \leq \kappa(x)$ for almost all $x\in \Omega$, then
$$ {\cal V}(t) \leq e^{-2 \kappa_0 t} {\cal V}(0).$$
\end{prop}
\begin{proof}
Integrating \eqref{eq:dialecttransport} with respect to ${\bf v}$ we find 
$$   \partial_t {\bf \overline{V}}(x,t)   =  - \kappa(x) {\bf \overline{V}}(x,t)  +     \int w(x,\tilde x)  f(\tilde x,{\bf \tilde v},t)   {\bf p}({\bf \tilde v})   ~d\tilde z 
. $$
Thus, 
\begin{align*}
& \frac{d}{dt} \frac{1}2 \int \int \vert{\bf v} - {\bf \overline{V}}(x,t) \vert^2 f(x,{\bf v},t)~dz \\ &\qquad = 
\frac{1}2 \int   \vert{\bf v} - {\bf \overline{V}}(x,t) \vert^2 \partial_t f(x,{\bf v},t)~dz 
- \int ({\bf v} - {\bf \overline{V}}(x,t))\cdot \partial_t {\bf \overline{V}}(x,t) f(x,{\bf v},t)~dz\\
	&\qquad= - \int \int w(x,\tilde x) ({\bf v} - {\bf \overline{V}}(x,t))\cdot f(x,{\bf v},t)   f(\tilde x,{\bf \tilde v},t) ( {\bf v} - {\bf p}({\bf \tilde v})) ~d\tilde z~dz +
	\\ &\qquad \qquad +  \int w(x,\tilde x)({\bf v} - {\bf \overline{V}}(x,t))\cdot (\kappa(x){\bf \overline{V}}(x,t) - \int   f(\tilde x,{\bf \tilde v},t)   {\bf p}({\bf \tilde v})   ~d\tilde z) f(x,{\bf v},t)~dz \\
	&\qquad = - \int \kappa(x) \vert{\bf v} - {\bf \overline{V}}(x,t) \vert^2   f(x,{\bf v},t)~dz.
\end{align*}
The assertions follow directly, respectively with Gronwall's lemma.
\end{proof}

Let us mention that analogous statements can be derived for other moments $p \geq$, those are equivalent for estimates of the metric
$$ d_p(f,\rho \delta_{\bf \overline{V}}) = \left(\int_\Omega W_p(f(x,\cdot),\rho(x) \delta_{{\bf \overline{V}}(x,\cdot)})^p~dx\right)^{1/p}, $$
with $W_p$ being the $p$-Wasserstein metric (taking into account the explicit form of Wasserstein metrics if one measure is concentrated).
Together with a stability estimate on solutions of \eqref{eq:dialecttransport}, we see that it can be expected that solutions are close to monokinetic ones, we leave a more quantitative analysis to future research.

\subsection{Spatially Local Approximation}

As a last step we consider the case of $w$ being a spatially local kernel, for simplicity we assume it is convolutional and moreover $\rho \equiv 1$ on a domain $\Omega \subset \R^d$. The local kernel is scaled such that
$$ w(x,\tilde x) = \epsilon^{-d} k(\epsilon^{-1} (x-\tilde x)) $$
and $k$ is assumed to be even. 
Then we find for a function $\varphi$ being smooth
$$ \int w(x,\tilde x) (\varphi(\tilde x) - \varphi(x)) ~dx = \frac{C \epsilon^2}2 \Delta \phi + {\cal O}(\epsilon^4), $$
with $C$ being the second moment of $k$. Using the notation $\sigma = \sqrt{C} \epsilon$ we obtain the following approximation for monokinetic solutions:
$$ \partial_t V = - \kappa (V - {\bf p}(V)) + \frac{\sigma^2}2 \Delta {\bf p}(V).$$
This is a nonlinear reaction-diffusion equation that was originally derived by Burridge \cite{burridge1}. For $M=2$  rigorous existence and uniqueness of classical and weak solutions (globally in time) can be shown (cf. \cite{doss}), for $M >2$ only local existence of classical solutions is known so far. Global existence and a quantitative analysis of the coarsening dynamics is a challenging open problem  due to the degenerate nonlinear cross-diffusion effects and the absence of a gradient flow structure. Let us mention that by introducing the the inverse function of ${\bf p}$, denoted by ${\bf V}$, we can equivalently formulate an equation for the vector ${\bf P}$ of probabilities
$$ \partial_t {\bf V}(P)  = - \kappa ({\bf V}(P) - P) + \frac{\sigma^2}2 \Delta  P.$$
 
From the derivation of the local equation we naturally expect $\kappa >> \sigma^2$. Using a time scaling such that $\sigma$ is of order one, we see that $\kappa$ is a large parameter, thus to leading order we have
$V = {\bf p}(V)$, so the approximation by the standard Allen-Cahn equation
$$ \partial_t V = - \kappa (V - {\bf p}(V)) + \frac{\sigma^2}2 \Delta V$$
respectively
$$ \partial_t P  = - \kappa ({\bf V}(P) - P) + \frac{\sigma^2}2 \Delta  P$$
may be equally accurate in the local limit. Note that the approximation for $V$ is the corresponding local approximation to the reaction-diffusion model \eqref{eq:dialectreactiondiffusion}, so at least in this scaling limit we expect the two modelling approaches to coincide.

\section{Case Study: Social Construction}

The paper by Shaw \cite{shaw} proposes an agent-based model of social learning, using a network of interaction between agents. In the model there are $M$ (in particular $M=4$ in \cite{shaw}) different mental representations of a social actions, each with a different weight $v_i$. In an interaction with another agent, who plays action $i$, the vector ${\bf \omega}$ of weights is updated via
$$ {\bf \omega}_i' = {\bf \omega} + e_i. $$
On the other hand, given a weight vector ${\bf w}$, the action with highest weight is played in the next interaction, respectively one of those with highest weigths is chosen with uniform probability if there are multiple ones. We can interpret this choice as a generalization of the probabilities $p_i$ in the dialect model above to a concentrated probability measure, it actually corresponds to the limit $\alpha \rightarrow 0$ in \eqref{eq:piburridge}. Moreover, the network interaction is rather discrete with $N$ agents and associated interaction weights $w_{k,\ell}$ between agents $k$ and $\ell$. 

In order to derive a Boltzmann-type  model we perform a suitable rescaling of the states from $w_i$ to 
$$ v_i = \frac{\omega_i}{\sum_{j=1}^M \omega_j} , $$
and  the number of interactions $I$ to $s = \frac{I}J$ for some reasonably large $J$.
The weight update in the interactions thus becomes
$${\bf v}_i' = \frac{s}{s+h} {\bf v} + \frac{h}{s+h} e_i, \quad s'=s+h $$
with $h=\frac{1}J$. The measure $\nu_{z,\tilde z}$ is given by
$$ \nu_{z,\tilde z} = \delta_0(da) \otimes \sum_{i=1}^M \delta_{b_i}(db) p_i({\bf \tilde v}) $$
with $b_i({\bf v},s) = (\frac{h}{s+h} (e_i-{\bf v},h)$.
 
Assuming again the existence of a single particle density $f_k({\bf v},S,t)$ on $\{1,\ldots,N\} \times K \times \R_+$ for $t > 0$  we obtain the Boltzmann equation in weak formulation as 
\begin{align*} 
&\frac{d}{dt} \int  \varphi( {\bf v},s) f_k({\bf v},s,t)~dz =   
\sum_i \sum_\ell w_{k \ell} \int \int    ( \varphi( {\bf v}_i',s')  - \varphi( {\bf v},s) ) p_i({\bf \tilde v}) f_k({\bf v},s,t)   f_\ell({\bf \tilde v},\tilde s,t)   d\tilde z~dz
\end{align*}
with $z=({\bf v},s)$.
Again existence and uniqueness of solutions can be shown by ODE arguments, in this case the density
$$ \rho_k = \int  f_k({\bf v},s,t) ~dz $$
is stationary. 


Using smallness of $h$ and rescaling time with $h$ we can derive the Vlasov  approximation 
\begin{align*} 
&\frac{d}{dt} \int  \varphi( {\bf v},s) f_k({\bf v},s,t)~dz = \\   & \qquad 
 \sum_\ell w_{k \ell} \int \int   \left( \nabla_{\bf v}  \varphi( {\bf v},s) )     \frac{1}{s} ({\bf p}({\bf \tilde v}) - {\bf v})  +   \partial_s \varphi( {\bf v},s) )  
\right)f_k({\bf v},s,t)   f_\ell({\bf \tilde v},\tilde s,t)~d\tilde z~dz.
\end{align*}
In strong form we obtain 
$$ \partial_t f_k({\bf v},s,t) + \nabla_{\bf v} \cdot \left( f_k({\bf v},s,t) \sum_\ell w_{k \ell}
\int  \frac{1}{s} ({\bf p}({\bf \tilde v}) - {\bf v}) f_\ell({\bf \tilde v},\tilde s,t)~d\tilde z \right)
+ \partial_s (f_k({\bf v},s,t) \lambda_k) = 0$$
with $ \lambda_k = \sum_\ell w_{k \ell} \rho_\ell $.
Possible monokinetic solutions are characterized by
$$
\frac{d{\bf V}_k}{dt} = \sum_\ell w_{k \ell}  \frac{1}{s_k} ({\bf p}({\bf V}_\ell) - {\bf V}_k), \qquad\frac{ds_k}{dt} = \lambda_k,
$$
which can be simplified to 
$$
\frac{d{\bf V}_k}{dt} = \sum_\ell w_{k \ell}  \frac{1}{s_k^0 + \lambda_k t} ({\bf p}({\bf V}_\ell) - {\bf V}_k).  
$$

The structure of the equation analogous to the dialect model above makes the phase separation and coarsening behaviour, i.e. the emergence of few social norms, quite clear. Due to the time-dependent weighting of the interactions it might be expected that equilibria are more dynamic however. 
Let us mention that using the discontinuous choice of ${\bf p}$ as in \cite{shaw} the existence and uniqueness of monokinetic solutions as well as of characteristics in the Vlasov equation cannot be shown easily and remains an interesting question for future research.

\section{Case Study: Pandemic Spread}

Modelling disease spread is nowadays a classical problem in applied mathematics (cf. \cite{capasso}), and in particular standard reaction and reaction-diffusion models are now a standard tool in epidemiology (cf. \cite{donofrio,smith}). However, nowadays diseases are spread by short-time travelers rather than by people moving to other locations. This is apparent in particular in the Covid-19 pandemic, where early infections in many countries and areas are due to short term travels (cf. e.g. \cite{badshah,correa,gudbjartsson,goessling}), so that human mobility networks may be a more relevant modelling structure (cf. e.g. \cite{bajardi,belik,kuchler,queiroz}).

In order to illustrate the effects we study a network-structured SIR-model in the following. This is a discrete state model with the three different states $S$ for susceptible, $I$ for infected, and $R$ for removed. There is only one pair interaction happening, namely between susceptible and infected, with the latter one being the active agent not changing its state, while the susceptible changes to infected. Moreover, the infected get removed at constant rate $\beta$. This yields.
\begin{align}
\partial_t \rho_S(x,t) &= - \int w(x,\tilde x) \rho_S(x,t) \rho_I(\tilde x,t) ~d\tilde x \label{eq:SIRnonlocal1a}\\
\partial_t \rho_I(x,t) &= \int w(x,\tilde x) \rho_S(x,t) \rho_I(\tilde x,t) ~d\tilde x - \beta \rho_I(x,t) \label{eq:SIRnonlocal2a}\\
\partial_t \rho_R &= \beta \rho_I
\end{align}

As usual in the SIR model we can ignore $\rho_R$ and simply consider the two-time-two system \eqref{eq:SIRnonlocal1a} and \eqref{eq:SIRnonlocal2a}. For brevity we denote the densities of susceptibles and infectives by $u$, $v$ instead of $\rho_S$, $\rho_I$. By introducing the nonlocal Laplacian nonlocal Laplacian
\begin{equation}
\Delta_w \varphi(x) = \int w(x,\tilde x) (\varphi(\tilde x) - \varphi(x))~dx. 
\end{equation} 
and using the notation $\alpha(x) = \int w(x,\tilde x) ~d\tilde x$  we can rewrite 
the system in a reaction-diffusion form 
\begin{align}
\partial_t u &= - \alpha u v - u \Delta_w v, \label{eq:SIRnonlocal1} \\ 
\partial_t v &= \alpha u v + u \Delta_w v - \beta v,  \label{eq:SIRnonlocal2}
\end{align}
This allows to give some comparison to the more standard reaction-diffusion models of epidemics, respectively their nonlocal version (cf. e.g. \cite{belik})
\begin{align}
\partial_t u &= - \alpha u v +D_1  \Delta_w u, \label{eq:SIRnonlocal1RD} \\ 
\partial_t v &= \alpha u v +  D_2 \Delta_w v - \beta v,  \label{eq:SIRnonlocal2RD}
\end{align}
The key difference is the linearity and non-degeneracy in the diffusion part, which induce a dispersal of both $u$ and $v$, while only $v$ disperses in the network-structured model. 

The behaviour of \eqref{eq:SIRnonlocal1}, \eqref{eq:SIRnonlocal2} is illustrated in Figure \ref{fig1} together with a comparison to the nonlocal reaction-diffusion model \eqref{eq:SIRnonlocal1RD}, \eqref{eq:SIRnonlocal2RD}. Those are based on a numerical solution of the models on the unit interval with periodicity, using the kernel 
$ w(x,y) = \alpha (x_0 - |x-x_0|)_+$ with $x_0=0.2$, $\alpha=0.3$, and $\beta =0.1$. The initial value of $u$ is constant equal to one, while the initial value of $v$ is a peak at $x=0.5$. The spatial grid size used is $h=0.01$ and the time step $\tau = 0.01$. We see that the overall dynamics in the two models is  similar, but the nonlocal reaction-diffusion model smoothes the peak in the infected population stronger (see time sequence of $u$ on the left), while the network structured model does not introduce a local peak in the susceptible one (see time sequence of $v$ on the right). As a consequence, the network structured model predicts a higher number of infected persons in the long run. 

\begin{figure}
\begin{center}
\includegraphics[width=0.35\textwidth]{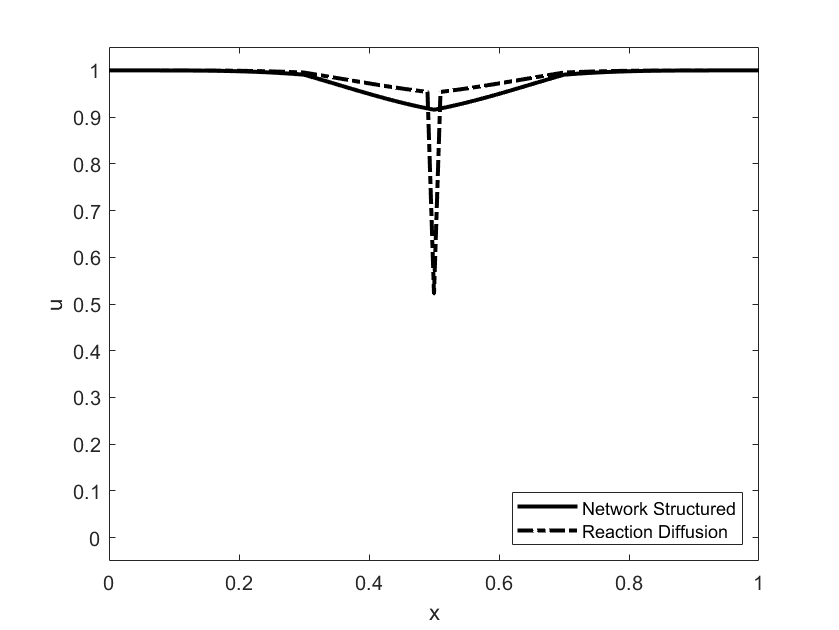} \includegraphics[width=0.35\textwidth]{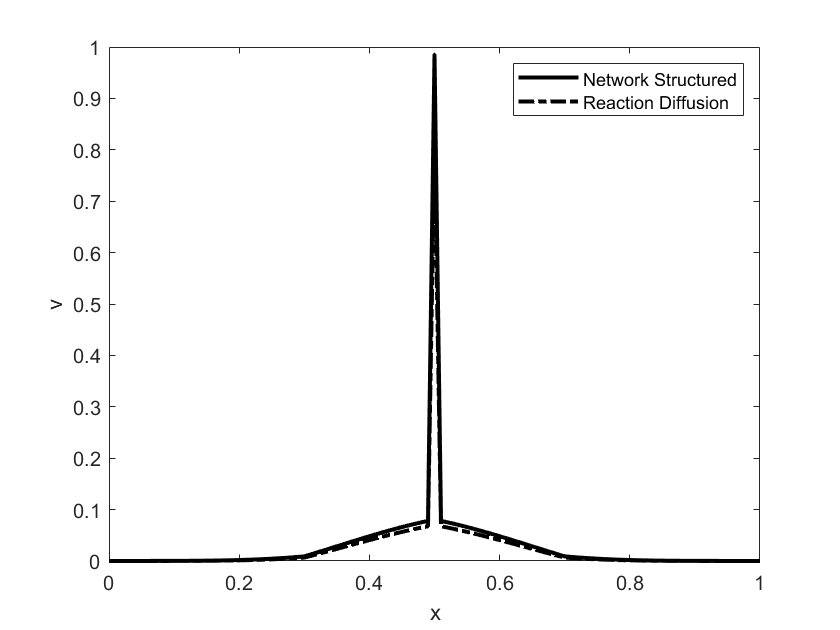}
\includegraphics[width=0.35\textwidth]{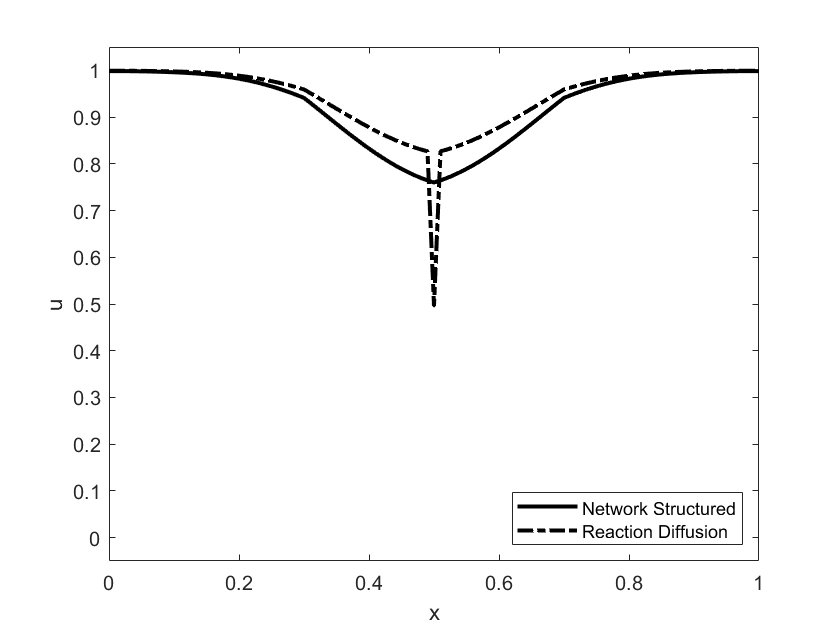} \includegraphics[width=0.35\textwidth]{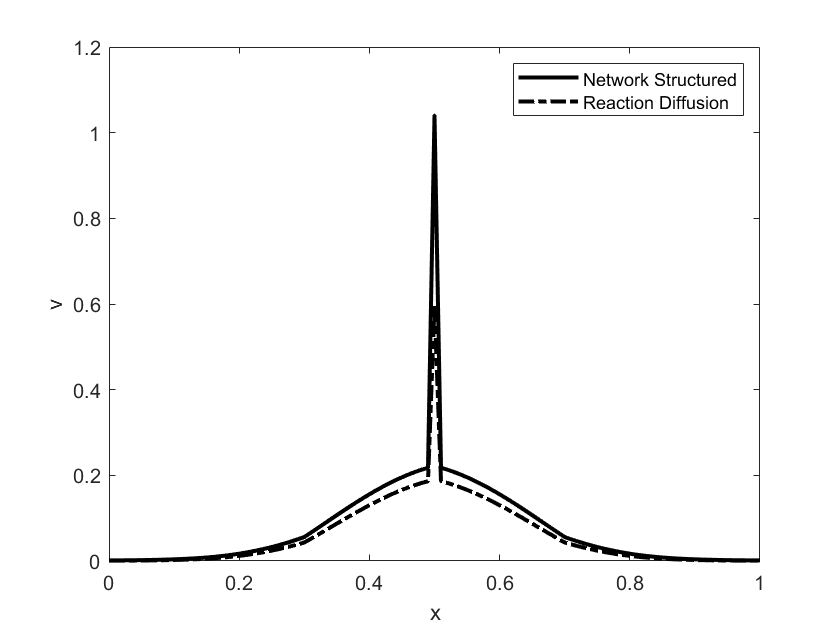}
\includegraphics[width=0.35\textwidth]{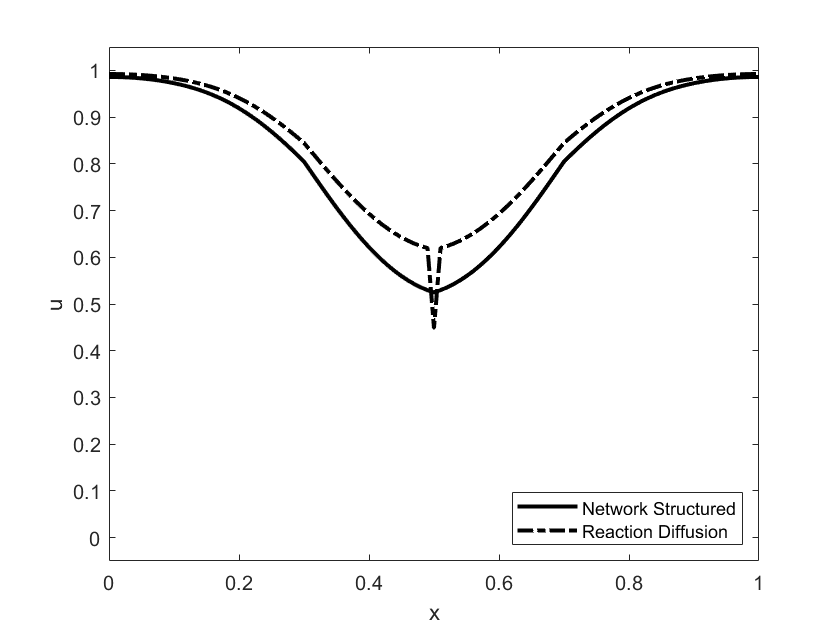} \includegraphics[width=0.35\textwidth]{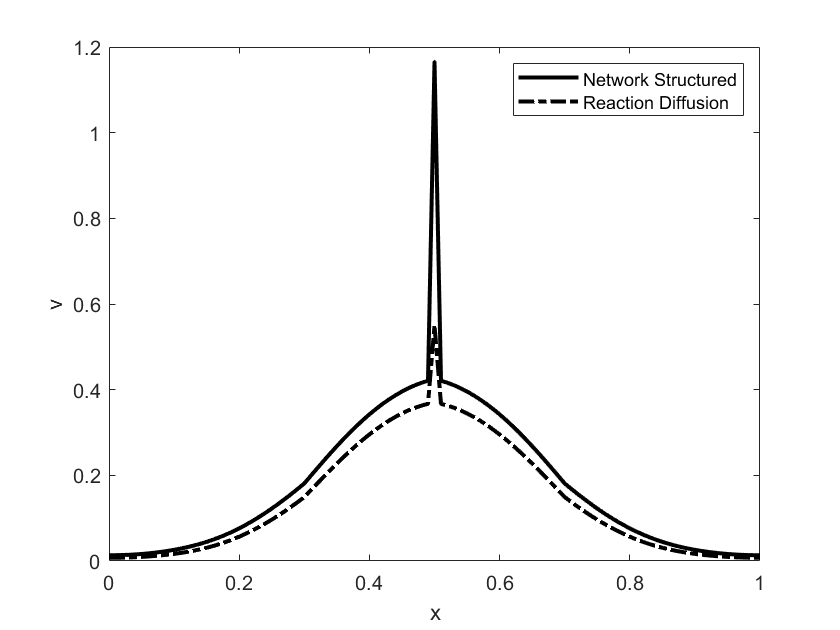}
\includegraphics[width=0.35\textwidth]{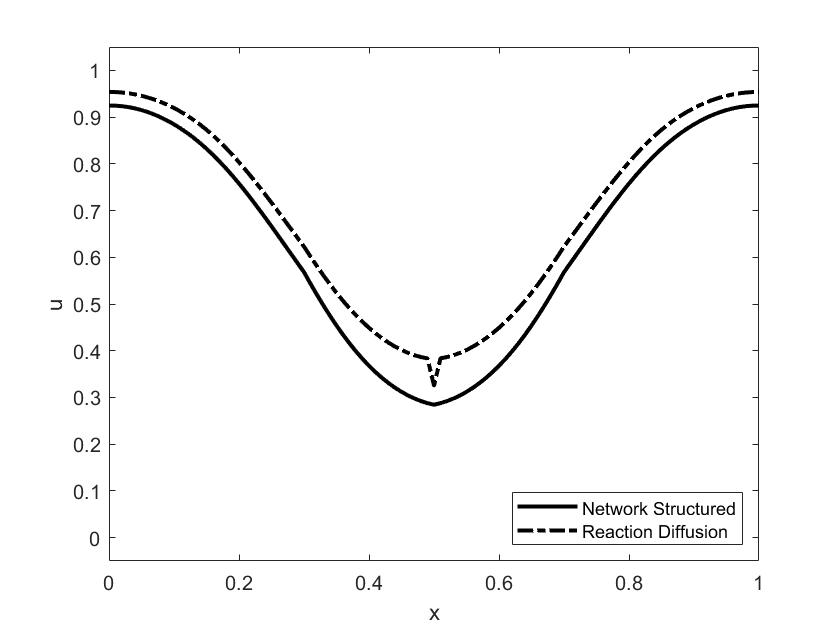} \includegraphics[width=0.35\textwidth]{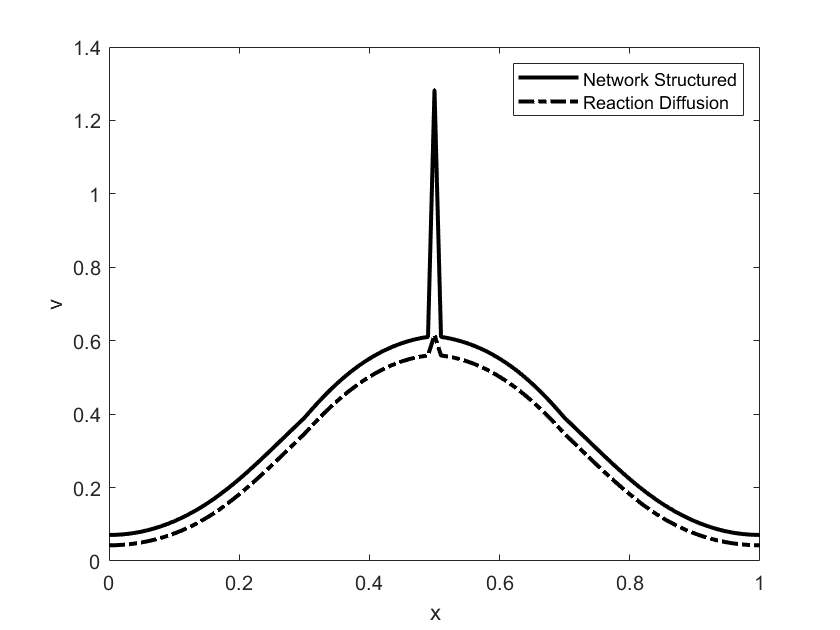}
\includegraphics[width=0.35\textwidth]{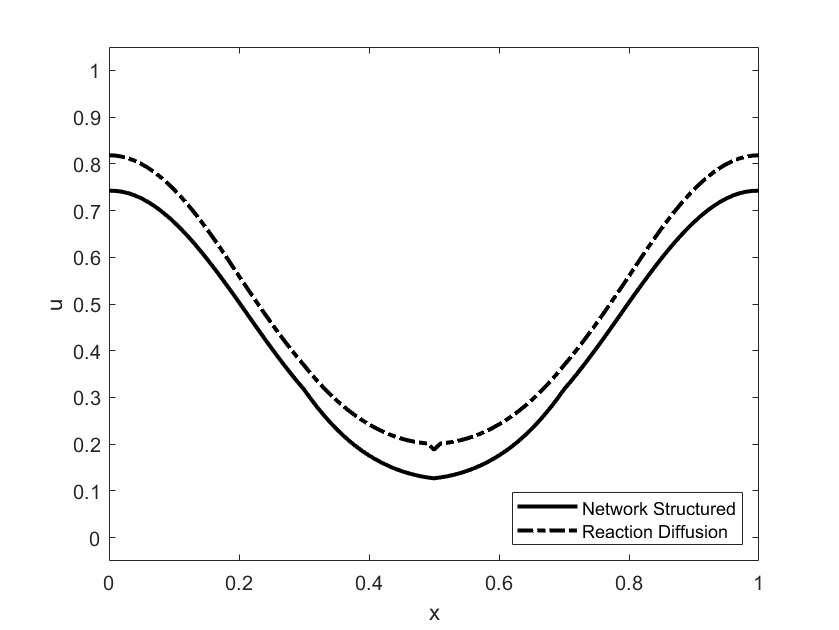} \includegraphics[width=0.35\textwidth]{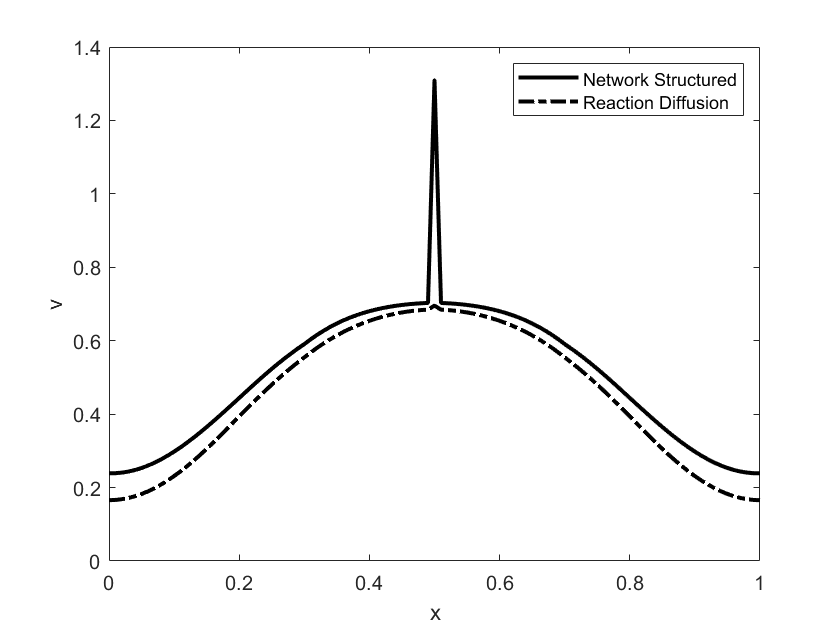}
\caption{Evolution of $u$ (left) and $v$ (right) at time steps $t=1,2,3,4,5$ with the network structured model in 
full lines and the nonlocal reaction-diffusion model in dash-dotted lines. \label{fig1}}
\end{center}
\end{figure}

Let $\Omega \subset \R^d$ be the maximal support of all involved functions, then we can confine the problem to $\Omega$ and provide a straightforward analysis:
 
\begin{prop}
Let $u^0 \in L^\infty(\Omega)$, $v^0 \in L^\infty(\Omega)$ be nonnegative initial values, and $w  \in L^\infty(\Omega;L^1(\Omega))$ be nonnegative. Then there exists a unique nonnegative solution $u \in C^1(0,T;L^\infty(\Omega))$, $v \in C^1(0,T;L^\infty(\Omega))$ of \eqref{eq:SIRnonlocal1}, \eqref{eq:SIRnonlocal2}. The solution satisfies 
$$ 0 \leq u(x,t) + v(x,t) \leq u_0(x) + v_0(x) $$
for almost every $x \in \Omega$ and every $t \in [0,T]$. Moreover, $\partial_t u $ is nonpositive almost everywhere. 
\end{prop}
\begin{proof}
The existence and uniqueness follows from a direct application of the Picard-Lindel\"of Theorem in $L^\infty(\Omega)^2$ (cf. \cite{brezis}), respectively a localized version. We apply the result first to obtain existence in the interval $[0,\tau]$ with  time step 
$$ \tau = \frac{1}{2 \Vert u_0 \Vert_{L^\infty(\Omega)}  \Vert w  \Vert_{L^\infty(\Omega;L^1(\Omega))}} $$ 
to establish existence of a solution in the invariant subset ${\cal I} \subset C(0,T;L^\infty(\Omega))$
$$ {\cal I} = \{ 0 \leq u \leq \Vert u_0 \Vert_\infty, 0 \leq e^{\beta t} v \leq 2 \Vert \tilde v_0 \Vert_{L^\infty(\Omega)}.$$ 
Since $\tau$ is uniform we can incrementally apply the same result to obtain existence and uniqueness of nonnegative solutions in an arbitrary interval $[0,T]$.
Finally nonpositivity of $\partial_t u$ follows from those of $u$ and $v$ directly from \eqref{eq:SIRnonlocal1}.
\end{proof}

As in the case of the dialect model, we can investigate the local limit, which, after appropriate scaling is of the form
\begin{align}
\partial_t u &= - \alpha u v - u \Delta  v \label{eq:SIR1local}\\
\partial_t v &= \alpha u v + u \Delta v - \beta v.  \label{eq:SIR2local}
\end{align}
The system \eqref{eq:SIR1local}, \eqref{eq:SIR2local} is a rather degenerate cross-diffusion system. It is easy to see that the operator $(u,v) \mapsto (u\Delta v,u \Delta v)$ respectively its linearization are 
not normally elliptic, hence the standard theory for parabolic systems (cf. e.g. \cite{amann}) does not apply. 
Even worse, we see that \eqref{eq:SIR1local} destroys some basic properties the model should naturally inherit, such as the nonpositivity of $\partial_t u$. If $v$ is locally concave such that  $\Delta v < - \alpha v$ this results in $ \partial_t u> 0 $, which contradicts the modelling assumption that the number of susceptibles cannot increase. The reason is that the local approximation beyond the leading order 
$$\partial_t u = - \alpha u v, \qquad \partial_t v = \alpha u v - \beta v $$
is justified only if the leading order solution is sufficiently small.

\section{Conclusions and Outlook}

In this paper we have derived network-structured kinetic equations and discussed their main properties, illustrated by some applications in human behaviour. We have demonstrated that challenging classes of (nonlocal) PDE systems can already arise as monokinetic equation of Vlasov approximations, further studies of the full model including a nontrivial variance in the state space are an interesting topic for future reseach. Our mainly formal approach also raises several further mathematical questions, e.g. the analysis of Vlasov and Fokker-Planck approximations, as well as the analysis of monokinetic equations and their local limits related to sparse graphs. Formal similarities to more standard models with explicit movement (or abstractly change in the structural variable), which we found in the local limit of monokinetialso raise a further question of 

A rather open topic is the derivation of macroscopic equations beyond monokinetic ones. Since there is no natural distinction into transport in $x$ and collision in $v$ as in standard kinetic models, the derivation of hydrodynamic equations cannot be based on asymptotics in the collision operators, not even by formal asymptotics as in the Hilbert or Chapman-Enskog expansion. The derivation of macroscopic equations is further impeded by the rather complicated and non-symmetric type of interactions found in behavioural sciences. 

An obvious question for extension of the models concerns the modification of the networks in time, which may become relevant e.g. for applications in social networks where the links are created or deleted on the same time scale as other processes like opinion formation. From a mathematical point of view it is a key issue to derive kinetic and macroscopic models including the full network structure, which seems a rather open problem.
 
Another aspect one may naturally ask from a mathematical point of view is the (optimal) control of network-structured problem, in the Boltzmann, Vlasov or monokinetic case. From an ethical point of view, this raises some issues however, e.g. when trying to control (or just influence) opinions on social networks via bots. In other cases control may be benefitial however, e.g. for avoiding pandemic spread or maybe also for counteracting the decrease of cultural diversity.  Similar control problems for finding consensus (cf. \cite{ruizbalet}) may also arise in mean-field models for robot swarms with a network communication structure (cf. \cite{elamvazhuthi}).

\section*{Acknowledgements}

The author acknowledges partial financial support by  European Union’s Horizon 2020 research and
innovation programme under the Marie Sk lodowska-Curie grant agreement No. 777826
(NoMADS) and the German Science Foundation (DFG) through CRC TR 154  "Mathematical Modelling, Simulation and Optimization Using the Example of Gas Networks".

\end{document}